\documentclass[12pt,a4paper,leqno]{amsart}

\usepackage[latin1]{inputenc}
\usepackage[T1]{fontenc}
\usepackage{amsfonts}
\usepackage{amsmath}
\usepackage{amssymb}
\usepackage{eurosym}
\usepackage{mathrsfs}
\usepackage{palatino}

\newcommand{\R}{\mathbb{R}}
\newcommand{\C}{\mathbb{C}}

\newcommand{\Z}{\mathbb{Z}}

\numberwithin{equation}{section}

\swapnumbers
\theoremstyle{plain}
\newtheorem{thm}[equation]{Theorem}
\newtheorem{lem}[equation]{Lemma}
\newtheorem{prop}[equation]{Proposition}
\newtheorem{cor}[equation]{Corollary}

\theoremstyle{definition}

\theoremstyle{remark}
\newtheorem{rem}[equation]{Remark}

\title{Representation of bi-parameter singular integrals by dyadic operators}

\author{Henri Martikainen}\thanks{The author is supported by the Academy of Finland through the project ``$L^p$ methods in harmonic analysis''.}
\address{Department of Mathematics and Statistics, University of Helsinki, P.O.B. 68, FI-00014 Helsinki, Finland}
\email{henri.martikainen@helsinki.fi}


\begin{document}
\maketitle

\begin{abstract}
We prove a dyadic representation theorem for bi-parameter singular integrals. That is, we represent certain bi-parameter operators
as rapidly decaying averages of what we call bi-parameter shifts. A new version of the product space $T1$ theorem is established as a consequence.
\end{abstract}

\section{Introduction}
We study certain bi-parameter singular integrals $T$ acting on some class of functions with product domain $\R^{n+m} = \R^n \times \R^m$. Our aim is to prove a representation theorem for them
as an average of bi-parameter shifts S:
\begin{displaymath}
\langle Tf, g\rangle = C_T \mathbb{E}_{w_n} \mathbb{E}_{w_m}\mathop{\sum_{(i_1, i_2) \in \Z_+^2}}_{(j_1, j_2) \in \Z_+^2} 2^{-\max(i_1,i_2)\delta/2}2^{-\max(j_1,j_2)\delta/2} \langle S^{i_1i_2j_1j_2}_{\mathcal{D}_n\mathcal{D}_m}f, g \rangle.
\end{displaymath}
Here the average is taken over all the dyadic grids $\mathcal{D}_n$ in $\R^n$ (parametrized by the random parameter $w_n$) and all the dyadic grids $\mathcal{D}_m$ in $\R^m$ (parametrized by the random parameter $w_m$).
An exact formulation of everything is given after the introduction.
Such a representation theorem exists for ordinary Calder\'on--Zygmund operators, and this was proven by Hyt\"onen \cite{Hy} in connection with the proof of the $A_2$ conjecture
for general singular integrals.

In the one-parameter case such general representation theorems have already been utilized several times after \cite{Hy}. The simplified proof of the $A_2$ conjecture by Hyt\"onen, P\'erez, Treil and Volberg \cite{HPTV}
offered among other things a bit easier formulation of the representation theorem. In \cite{HLMORSUT} the author together with Hyt\"onen, Lacey, Orponen, Reguera, Sawyer and Uriarte--Tuero used
the representation theorem to study sharp weak and strong type weighted bounds for maximal truncations $T_{\#}$.
Modifying the metric randomization by Hyt\"onen and the author \cite{HM} these representation theorems were lifted to the generality of metric spaces by Nazarov, Reznikov and Volberg \cite{NRV}. Several other applications in the weighted context
also already exist.

The reason why the representation theorem is so useful in the one-parameter case is that it can be used to reduce problems considering a general singular integral $T$ into purely dyadic problems considering shifts only. Because of this, there is no particular reason why the applications
should be limited to weighted questions. This just happens to be the case, since the representation theorem was originally developed for this purpose and is still very new a result. This is motivation enough for us to develop the analogous theory in the bi-parameter case. It would, of course, be interesting to study sharp weighted theory in the bi-parameter setting. Our theorem might be useful for this, however, it is a very difficult problem.

Regarding multi-parameter singular integrals, and multi-parameter harmonic analysis in general, there is a very large existing theory. After the classical $T1$ and $Tb$ type theory by David and Journ\'e \cite{DJ} and
David, Journ\'e and Semmes \cite{DJS}, the first $T1$ type theorem for product spaces was proved by Journ\'e \cite{Jo}. Regarding other classical theory, we only mention the work of
Chang and Fefferman \cite{CF}, Fefferman \cite{Fe} and Fefferman and Stein \cite{FS}. These three concern singular integrals and various spaces, like the BMO, on the product setting.
There is a wide body of more recent developments of which we here only mention the papers by Ferguson and Lacey \cite{FL}, Lacey and Metcalfe \cite{LM} and
Muscalu, Pipher, Tao and Thiele \cite{MPTT}. These have to do with various multi-parameter paraproducts and characterizations for some product spaces. Some bi-parameter paraproducts appear also in our proof, and the
product BMO space is thus important for us.

The classical multi-parameter singular integral theory of Journ\'e \cite{Jo} involves formulations written in the language of vector-valued Calder\'on--Zygmund theory. Very recently Pott and Villarroya \cite{PV} formulated and proved
a new type of $T1$ theorem for product spaces. There such vector-valued formulations are replaced by several new mixed type conditions. Here we define our bi-parameter operators inspired by \cite{PV}. The conditions we use are not
exactly the same. We, for example, do not work with smooth testing conditions.
Establishing the correct shift structure is our primary task. However, we do get, as a by product, a pretty nice form of the product space $T1$ theorem.

In this paper we bring the superbly useful machinery of non-homogeneous analysis pioneered by Nazarov, Treil and Volberg (see for example \cite{NTV}) to the context of bi-parameter theory. The use of non-homogeneous analysis
gives additional decay for certain matrix elements involved in the expansion of $\langle Tf, g\rangle$. Just like in Hyt\"onen's proof of the representation theorem for one-parameter singular integrals, the proof is a $T1$ style proof with ingredients
from non-homogeneous analysis. In our case, we have to deal with the much added complexity of the bi-parameter situation.
Indeed, there are more cases than in the one-parameter setting, and many of these are interesting mixed type phenomena. The non-homogeneous analysis makes this splitting into cases nicely transparent getting rid of rare
geometric complications.

\subsection*{Acknowledgements}
The author wishes to thank Tuomas Hyt\"onen for suggesting this topic and for useful discussions.

\section{Definitions, strategy and the main result}
\subsection*{Structural assumptions}
Let us formulate the Calder\'on--Zygmund structure of our operators.
The basic assumption is that if $f = f_1 \otimes f_2$ (meaning $f(x) = f_1(x_1)f_2(x_2)$ for $x=(x_1,x_2)$) and $g= g_1 \otimes g_2$
with $f_1, g_1 \colon \R^n \to \C$, $f_2, g_2 \colon \R^m \to \C$, $\textrm{spt}\,f_1\cap \textrm{spt}\,g_1 = \emptyset$ and $ \textrm{spt}\,f_2 \cap  \textrm{spt}\,g_2 = \emptyset$, then
we have the kernel representation
\begin{displaymath}
\langle Tf, g\rangle = \int_{\R^{n+m}}\int_{\R^{n+m}} K(x,y)f(y)g(x)\,dx\,dy.
\end{displaymath}
The kernel  $K\colon (\R^{n+m} \times \R^{n+m}) \setminus \{(x,y) \in \R^{n+m}  \times \R^{n+m}:\, x_1 = y_1 \textrm{ or } x_2 = y_2\} \to \C$ is assumed to satisfy the size condition
\begin{displaymath}
|K(x,y)| \le C\frac{1}{|x_1-y_1|^{n}}\frac{1}{|x_2-y_2|^{m}}
\end{displaymath}
and the H\"older conditions
\begin{align*}
|K(x,y) - K(x, &(y_1, y_2')) - K(x, (y_1', y_2)) + K(x, y')| \\
&\le  C \frac{|y_1-y_1'|^{\delta}}{|x_1-y_1|^{n+\delta}}  \frac{|y_2-y_2'|^{\delta}}{|x_2-y_2|^{m+\delta}} 
\end{align*}
whenever $|y_1-y_1'| \le |x_1-y_1|/2$ and $|y_2-y_2'| \le |x_2-y_2|/2$,
\begin{align*}
|K(x,y) - K((x_1, x_2'), &y) - K((x_1', x_2), y) + K(x', y)| \\
&\le C\frac{|x_1-x_1'|^{\delta}}{|x_1-y_1|^{n+\delta}}  \frac{|x_2-x_2'|^{\delta}}{|x_2-y_2|^{m+\delta}} 
\end{align*}
whenever $|x_1 - x_1'| \le |x_1-y_1|/2$ and $|x_2-x_2'| \le |x_2-y_2|/2$,
\begin{align*}
|K(x,y) - K((x_1,x_2'), &y) - K(x, (y_1', y_2)) + K((x_1,x_2'), (y_1', y_2))| \\
&\le  C \frac{|y_1-y_1'|^{\delta}}{|x_1-y_1|^{n+\delta}}  \frac{|x_2-x_2'|^{\delta}}{|x_2-y_2|^{m+\delta}} 
\end{align*}
whenever $|y_1-y_1'| \le |x_1-y_1|/2$ and $|x_2-x_2'| \le |x_2-y_2|/2$, and
\begin{align*}
|K(x,y) - K(x, &(y_1,y_2')) - K((x_1', x_2), y) + K((x_1',x_2), (y_1, y_2'))| \\
&\le  C \frac{|x_1-x_1'|^{\delta}}{|x_1-y_1|^{n+\delta}}  \frac{|y_2-y_2'|^{\delta}}{|x_2-y_2|^{m+\delta}} 
\end{align*}
whenever $|x_1 - x_1'| \le |x_1-y_1|/2$ and $|y_2-y_2'| \le |x_2-y_2|/2$.

Furthermore, we assume the mixed H\"older and size conditions
\begin{align*}
|K(x,y) - K((x_1', x_2), y)| \le C\frac{|x_1-x_1'|^{\delta}}{|x_1-y_1|^{n+\delta}} \frac{1}{|x_2-y_2|^m}
\end{align*}
whenever $|x_1-x_1'| \le |x_1-y_1|/2$,
\begin{align*}
|K(x,y) - K(x, (y_1', y_2))| \le C\frac{|y_1-y_1'|^{\delta}}{|x_1-y_1|^{n+\delta}} \frac{1}{|x_2-y_2|^m}
\end{align*}
whenever $|y_1-y_1'| \le |x_1-y_1|/2$,
\begin{align*}
|K(x,y) - K((x_1, x_2'), y)| \le C \frac{1}{|x_1-y_1|^n}\frac{|x_2-x_2'|^{\delta}}{|x_2-y_2|^{m+\delta}}
\end{align*}
whenever $|x_2-x_2'| \le |x_2-y_2|/2$, and
\begin{align*}
|K(x,y) - K(x, (y_1,y_2'))| \le C \frac{1}{|x_1-y_1|^n}\frac{|y_2-y_2'|^{\delta}}{|x_2-y_2|^{m+\delta}}
\end{align*}
whenever $|y_2-y_2'| \le |x_2-y_2|/2$. We use, for minor convenience, $\ell^{\infty}$ metrics on $\R^n$ and $\R^m$.

We also need some Calder\'on--Zygmund structure on $\R^n$ and $\R^m$ separately. If $f = f_1 \otimes f_2$ and $g = g_1 \otimes g_2$ with $\textrm{spt}\,f_1\cap \textrm{spt}\,g_1 = \emptyset$, then we assume the kernel representation
\begin{displaymath}
\langle Tf, g\rangle = \int_{\R^n} \int_{\R^n} K_{f_2, g_2}(x_1,y_1)f_1(y_1)g_1(x_1)\,dx_1\,dy_1.
\end{displaymath}
The kernel $K_{f_2, g_2} \colon (\R^n \times \R^n) \setminus \{(x_1, y_1) \in \R^n \times \R^n:\, x_1 = y_1\}$ is assumed to satisfy the size condition
\begin{displaymath}
|K_{f_2,g_2}(x_1,y_1)| \le C(f_2,g_2)\frac{1}{|x_1-y_1|^{n}}
\end{displaymath}
and the H\"older conditions
\begin{displaymath}
|K_{f_2,g_2}(x_1,y_1) - K_{f_2,g_2}(x_1',y_1)| \le C(f_2,g_2)\frac{|x_1-x_1'|^{\delta}}{|x_1-y_1|^{n+\delta}}
\end{displaymath}
whenever $|x_1 - x_1'| \le |x_1-y_1|/2$, and
\begin{displaymath}
|K_{f_2,g_2}(x_1,y_1) - K_{f_2,g_2}(x_1,y_1')| \le C(f_2,g_2)\frac{|y_1-y_1'|^{\delta}}{|x_1-y_1|^{n+\delta}}
\end{displaymath}
whenever $|y_1 - y_1'| \le |x_1-y_1|/2$. Let $|A|$ denote the Lebesgue measure of a set $A$ and $\chi_A$ be the characteristic function of $A$. We need the above representations and some control for $C(f_2,g_2)$ only in the diagonal in the following sense. 
For every cube $V \subset \R^m$ we assume that there holds $C(\chi_V, \chi_V) + C(\chi_V, u_V) + C(u_V, \chi_V) \le C|V|$, whenever $u_V$ is such a function that spt$\,u_V \subset V$, $|u_V| \le 1$ and $\int u_V = 0$. Functions $u_V$ are called $V$-adapted with zero-mean (so
$V$-adapted means just the first two conditions on the support and size). 
We also assume the analogous representation and properties with a kernel $K_{f_1, g_1}$ in the case spt$\,f_2 \cap \textrm{spt}\, g_2 = \emptyset$.

\subsection*{Boundedness and cancellation assumptions}
Define the partial adjoint $T_1$ of $T$ by setting
\begin{displaymath}
\langle T_1(f_1 \otimes f_2), g_1 \otimes g_2 \rangle = \langle T(g_1 \otimes f_2), f_1 \otimes g_2\rangle.
\end{displaymath}
We assume that $T1, T^*1, T_1(1)$ and $T_1^*(1)$ belong to the product BMO on $\R^n \times \R^m$. We recall the definition of this space later in this section.

We assume that $|\langle T(\chi_K \otimes \chi_V), \chi_K \otimes \chi_V\rangle| \le C|K||V|$ for every cube $K \subset \R^n$ and $V \subset \R^m$. This is the weak boundedness property for $T$.

We also assume the following diagonal BMO conditions: for every cube $K \subset \R^n$ and $V \subset \R^m$ and for every zero-mean functions $a_K$ and $b_V$ which are $K$ and $V$ adapted respectively
(one has spt$\, a_K \subset K$, $|a_K| \le 1$ and $\int a_K = 0$, and similarly for $b_V$):
\begin{itemize}
\item[(i)] $|\langle T(a_K \otimes \chi_V), \chi_K \otimes \chi_V\rangle| \le C|K||V|$,
\item[(ii)] $|\langle T(\chi_K \otimes \chi_V), a_K \otimes \chi_V\rangle| \le C|K||V|$,
\item[(iii)] $|\langle T(\chi_K \otimes b_V), \chi_K \otimes \chi_V\rangle| \le C|K||V|$,
\item[(iv)] $|\langle T(\chi_K \otimes \chi_V), \chi_K \otimes b_V\rangle| \le C|K||V|$.
\end{itemize}

\subsection*{Haar functions} Let $h_I$ be a $L^2$ normalized Haar function related to $I \in \mathcal{D}_n$, where $\mathcal{D}_n$ is a dyadic grid on $\R^n$.
With this we mean that $h_I$, $I = I_1 \times \cdots \times I_n$, is one of the $2^n$ functions $h_I^{\eta}$, $\eta = (\eta_1, \ldots, \eta_n) \in \{0,1\}^n$, defined by
\begin{displaymath}
h_I^{\eta} = h_{I_1}^{\eta_1} \otimes \cdots \otimes h_{I_n}^{\eta_n}, 
\end{displaymath}
where $h_{I_i}^0 = |I_i|^{-1/2}\chi_{I_i}$ and $h_{I_i}^1 = |I_i|^{-1/2}(\chi_{I_{i, l}} - \chi_{I_{i, r}})$ for every $i = 1, \ldots, n$. Here $I_{i,l}$ and $I_{i,r}$ are the left and right
halves of the interval $I_i$ respectively. If $\eta \ne 0$ the Haar function is cancellative: $\int h_I = 0$. All the cancellative Haar functions form an orthonormal basis of $L^2(\R^n)$.
If $a \in L^2(\R^n)$ we may thus write $a = \sum_{I \in \mathcal{D}_n} \sum_{\eta \in \{0,1\}^n \setminus \{0\}} \langle a, h_I^{\eta} \rangle h_I^{\eta}$. However, we suppress the finite $\eta$ summation
and just write $a = \sum_I \langle a, h_I\rangle h_I$. Given a dyadic grid $\mathcal{D}_m$ on $\R^m$ and a cube $J \in \mathcal{D}_m$, we denote a $L^2$ normalized Haar function on $J$ by $u_J$.

\subsection*{Product BMO on $\R^n \times \R^m$}
Let us be given a dyadic grid $\mathcal{D}_n$ in $\R^n$ and a dyadic grid $\mathcal{D}_m$ in $\R^m$. We define the square function
\begin{displaymath}
S_{\mathcal{D}_n\mathcal{D}_m}f = \Big[ \sum_{K \in \mathcal{D}_n} \sum_{V \in \mathcal{D}_m} |\langle f, h_K \otimes u_V \rangle|^2 \frac{\chi_K \otimes \chi_V}{|K||V|} \Big]^{1/2}.
\end{displaymath}
Then the product Hardy space $H^1_{\mathcal{D}_n\mathcal{D}_m}(\R^n \times \R^m)$ consists of the locally integrable functions $f$ with $\|f\|_{H^1_{\mathcal{D}_n\mathcal{D}_m}(\R^n \times \R^m)} = \|S_{\mathcal{D}_n\mathcal{D}_m}f\|_1 < \infty$. The dual of this space
is the product BMO space BMO$_{\mathcal{D}_n\mathcal{D}_m}(\R^n \times \R^m)$.

For us, the condition that $b \in \{T1, T^*1, T_1(1), T_1^*(1)\}$ is in the product BMO is defined to mean that $\|b\|_{\textup{BMO}_{\mathcal{D}_n\mathcal{D}_m}(\R^n \times \R^m)} \le C$ with every dyadic grid $\mathcal{D}_n$ in $\R^n$ and every dyadic grid $\mathcal{D}_m$ in $\R^m$. 

\subsection*{Bi-parameter shifts}
A bi-parameter shift on $\R^n \times \R^m$ is tied to a dyadic grid $\mathcal{D}_n$ on $\R^n$, a dyadic grid $\mathcal{D}_m$ on $\R^m$ and non-negative integers $i_1, i_2, j_1, j_2$.
Such an operator is denoted by $S^{i_1i_2j_1j_2}_{\mathcal{D}_n\mathcal{D}_m}$ and is of the form
\begin{displaymath}
S^{i_1i_2j_1j_2}_{\mathcal{D}_n\mathcal{D}_m}f = \sum_{K \in \mathcal{D}_n} \sum_{V \in \mathcal{D}_m} A^{i_1i_2j_1j_2}_{KV}f,
\end{displaymath}
where
\begin{displaymath}
A^{i_1i_2j_1j_2}_{KV}f = \mathop{\mathop{\sum_{I_1,\,I_2 \subset K}}_{\ell(I_1) = 2^{-i_1}\ell(K)}}_{\ell(I_2) = 2^{-i_2}\ell(K)}
\mathop{\mathop{\sum_{J_1, \,J_2 \subset V}}_{\ell(J_1) = 2^{-j_1}\ell(V)}}_{\ell(J_2) = 2^{-j_2}\ell(V)}
 a_{I_1I_2KJ_1J_2V} \langle f, h_{I_1} \otimes u_{J_1} \rangle h_{I_2} \otimes u_{J_2}
\end{displaymath}
with
\begin{displaymath}
| a_{I_1I_2KJ_1J_2V}| \le \frac{|I_1|^{1/2}|I_2|^{1/2}}{|K|}\frac{|J_1|^{1/2}|J_2|^{1/2}}{|V|}.
\end{displaymath}
Here, of course, $I_1, I_2 \in \mathcal{D}_n$ and $J_1, J_2 \in \mathcal{D}_m$, and $\ell(I)$ denotes the side length of a cube $I$. It is also required that all the subshifts
\begin{displaymath}
S^{i_1i_2j_1j_2}_{\mathcal{A}\mathcal{B}} = \sum_{K \in \mathcal{A}} \sum_{V \in \mathcal{B}} A^{i_1i_2j_1j_2}_{KV}f, \qquad \mathcal{A} \subset \mathcal{D}_n, \, \mathcal{B} \subset \mathcal{D}_m,
\end{displaymath}
map $L^2(\R^n \times \R^m) \to L^2(\R^n \times \R^m)$ with norm at most one. If all of the Haar functions $h_{I_1}, h_{I_2}, u_{J_1}, u_{J_2}$ appearing are cancellative, the shift is called cancellative. Otherwise, it is called non-cancellative. The last requirement
concerning the $L^2$ boundedness of all of the subshifts follows from the other conditions for cancellative shifts.

In practice, it is useful to observe that a bi-parameter shift $S$ of type $(i_1, i_2, j_1, j_2)$ related to some dyadic grids is simply of the form
\begin{align*}
Sf(x) = \sum_{K,V} A_{KV}f(x) &= \sum_{K,V} \frac{1}{|K \times V|} \int_{K \times V} K_{A_{KV}}(x,y)f(y)\,dy \\ &= \int_{\R^{n+m}} K_S(x,y)f(y)\,dy,
\end{align*}
where first of all spt$\, K_{A_{KV}} \subset (K \times V) \times (K \times V)$ and $|K_{A_{KV}}(x,y)| \le 1$. Moreover, $K_{A_{KV}}$ is constant with respect to $x$
on dyadic rectangles $I \times J \subset K \times V$ for which $\ell(I) < 2^{-i_2}\ell(K)$ and $\ell(J) < 2^{-j_2}\ell(V)$, and $K_{A_{KV}}$ is constant with respect to $y$
on dyadic rectangles $I \times J \subset K \times V$ for which $\ell(I) < 2^{-i_1}\ell(K)$ and $\ell(J) < 2^{-j_1}\ell(V)$. Note also that clearly
\begin{displaymath}
|K_S(x,y)| \le C\frac{1}{|x_1-y_1|^n} \frac{1}{|x_2-y_2|^m}.
\end{displaymath}

\subsection{Random dyadic grids and the basic averaging formula}
Let $w_n = (w_n^i)_{i \in \Z}$ and $w_m =  (w_m^j)_{j \in \Z}$, where $w_n^i \in \{0,1\}^n$ and $w_m^j \in \{0,1\}^m$. Let $\mathcal{D}^0_n$ and $\mathcal{D}^0_m$ be the standard dyadic grids on $\R^n$ and $\R^m$ respectively.
In $\R^n$ we define the new dyadic grid $\mathcal{D}_n = \{I + \sum_{i:\, 2^{-i} < \ell(I)} 2^{-i}w_n^i: \, I \in \mathcal{D}_n^0\} = \{I + w_n: \, I \in \mathcal{D}_n^0\}$, where we simply have defined
$I + w_n := I + \sum_{i:\, 2^{-i} < \ell(I)} 2^{-i}w_n^i$. The dyadic grid $\mathcal{D}_m$ in $\R^m$ is similarly defined. There is a natural product probability structure on $(\{0,1\}^n)^{\Z}$ and $(\{0,1\}^m)^{\Z}$. So we have independent
random dyadic grids $\mathcal{D}_n$ and $\mathcal{D}_m$ in $\R^n$ and $\R^m$ respectively. Even if $n=m$ we need two independent grids.

A cube $I \in \mathcal{D}_n$ is called bad if there exists $\tilde I \in \mathcal{D}_n$ so that $\ell(\tilde I) \ge 2^r \ell(I)$ and $d(I, \partial \tilde I) \le 2\ell(I)^{\gamma_n}\ell(\tilde I)^{1-\gamma_n}$.
Here $\gamma_n = \delta/(2n + 2\delta)$, where $\delta > 0$ appears in the kernel estimates. One notes that
$\pi_{\textrm{good}}^n := \mathbb{P}_{w_n}(I + w_n \textrm{ is good})$ is independent of $I \in \mathcal{D}^0_n$. The parameter $r$ is a fixed constant so that $\pi_{\textrm{good}}^n, \pi_{\textrm{good}}^m > 0$.
Furthermore, it is important to note that for a fixed $I \in \mathcal{D}^0_n$
the set $I + w_n$ depends on $w_n^i$ with $2^{-i} < \ell(I)$, while the goodness of $I + w_n$ depends on $w_n^i$ with $2^{-i} \ge \ell(I)$. In particular, these notions are independent.
Analogous definitions and remarks related to $\mathcal{D}_m$ hold.

We prove the basic averaging formula of Hyt\"onen \cite{Hy} but in the bi-parameter setting. This is the only part of the proof where probabilistic arguments are needed, and here independence plays a big role, even more so in the bi-parameter setting.
We note that the functions $f$ and $g$ in this paper are always taken from some particularly nice dense subset of functions.
\begin{prop}
There holds
\begin{align*}
\langle Tf, g\rangle = C\mathbb{E} \sum_{I_1, I_2 \in \mathcal{D}_n} \sum_{J_1, J_2 \in \mathcal{D}_m}
&\chi_{\textrm{good}}(\textrm{smaller}(I_1,I_2)) \chi_{\textrm{good}}(\textrm{smaller}(J_1,J_2)) \\
& \langle T(h_{I_1} \otimes u_{J_1}), h_{I_2} \otimes u_{J_2}\rangle \langle f, h_{I_1} \otimes u_{J_1}\rangle \langle g, h_{I_2} \otimes u_{J_2}\rangle,
\end{align*}
where $\mathbb{E} = \mathbb{E}_{w_n} \mathbb{E}_{w_m}$ and $C = 1/(\pi_{\textrm{good}}^n \pi_{\textrm{good}}^m)$.
\end{prop}
\begin{rem}
Here all the appearing Haar functions are, of course, cancellative
and we recall that the finite summations over the $2^n-1$ or $2^m-1$ different cancellative Haar functions per cube are simply suppressed from the notation.
\end{rem}
\begin{proof}
Define $\langle f, h_I\rangle_1(y) = \int f(x,y)h_I(x)\,dx$, $y \in \R^m$. We may write
\begin{displaymath}
f = \sum_{I_1 \in \mathcal{D}_n} h_{I_1} \otimes \langle f, h_{I_1}\rangle_1 =  \sum_{I_1 \in \mathcal{D}_n^0} h_{I_1 + w_n} \otimes \langle f, h_{I_1 + w_n}\rangle_1
\end{displaymath}
so that by independence
\begin{align*}
\langle Tf, g\rangle &= E_{w_n} \sum_{I_1 \in \mathcal{D}_n^0} \langle T(h_{I_1 + w_n} \otimes \langle f, h_{I_1 + w_n}\rangle_1), g \rangle \\
&= \frac{1}{\pi_{\textrm{good}}^n}E_{w_n}   \sum_{I_1 \in \mathcal{D}_n^0} \chi_{\textrm{good}}(I_1 + w_n) \langle T(h_{I_1 + w_n} \otimes \langle f, h_{I_1 + w_n}\rangle_1), g \rangle.
\end{align*}
After expanding $g$ similarly as $f$ above, one sees that this equals
\begin{align*}
&\frac{1}{\pi_{\textrm{good}}^n}E_{w_n}   \sum_{I_1, I_2 \in \mathcal{D}_n^0} \chi_{\textrm{good}}(I_1 + w_n) \langle T(h_{I_1 + w_n} \otimes \langle f, h_{I_1 + w_n}\rangle_1), h_{I_2 + w_n} \otimes  \langle g, h_{I_2 + w_n}\rangle_1 \rangle \\
&= \frac{1}{\pi_{\textrm{good}}^n}E_{w_n}  \mathop{\sum_{I_1, I_2 \in \mathcal{D}_n^0}}_{\ell(I_1) \le \ell(I_2)} \chi_{\textrm{good}}(I_1 + w_n) \langle T(h_{I_1 + w_n} \otimes \langle f, h_{I_1 + w_n}\rangle_1), h_{I_2 + w_n} \otimes  \langle g, h_{I_2 + w_n}\rangle_1 \rangle \\
&+ E_{w_n}  \mathop{\sum_{I_1, I_2 \in \mathcal{D}_n^0}}_{\ell(I_1) > \ell(I_2)} \langle T(h_{I_1 + w_n} \otimes \langle f, h_{I_1 + w_n}\rangle_1), h_{I_2 + w_n} \otimes  \langle g, h_{I_2 + w_n}\rangle_1 \rangle.
\end{align*}
Here we again used independence in the latter summation. Comparing to the trivial representation
\begin{displaymath}
\langle Tf, g\rangle = E_{w_n}   \sum_{I_1, I_2 \in \mathcal{D}_n^0} \langle T(h_{I_1 + w_n} \otimes \langle f, h_{I_1 + w_n}\rangle_1), h_{I_2 + w_n} \otimes  \langle g, h_{I_2 + w_n}\rangle_1 \rangle 
\end{displaymath}
we conclude that
\begin{align*}
&\pi_{\textrm{good}}^n E_{w_n}  \mathop{\sum_{I_1, I_2 \in \mathcal{D}_n^0}}_{\ell(I_1) \le \ell(I_2)} \langle T(h_{I_1 + w_n} \otimes \langle f, h_{I_1 + w_n}\rangle_1), h_{I_2 + w_n} \otimes  \langle g, h_{I_2 + w_n}\rangle_1 \rangle  \\
&= E_{w_n}  \mathop{\sum_{I_1, I_2 \in \mathcal{D}_n^0}}_{\ell(I_1) \le \ell(I_2)} \chi_{\textrm{good}}(I_1 + w_n) \langle T(h_{I_1 + w_n} \otimes \langle f, h_{I_1 + w_n}\rangle_1), h_{I_2 + w_n} \otimes  \langle g, h_{I_2 + w_n}\rangle_1 \rangle.
\end{align*}
First expanding $g$ and proceeding like above one gets the symmetric formula
\begin{align*}
&\pi_{\textrm{good}}^nE_{w_n}  \mathop{\sum_{I_1, I_2 \in \mathcal{D}_n^0}}_{\ell(I_2) < \ell(I_1)} \langle T(h_{I_1 + w_n} \otimes \langle f, h_{I_1 + w_n}\rangle_1), h_{I_2 + w_n} \otimes  \langle g, h_{I_2 + w_n}\rangle_1 \rangle  \\
&= E_{w_n}  \mathop{\sum_{I_1, I_2 \in \mathcal{D}_n^0}}_{\ell(I_2) < \ell(I_1)} \chi_{\textrm{good}}(I_2 + w_n) \langle T(h_{I_1 + w_n} \otimes \langle f, h_{I_1 + w_n}\rangle_1), h_{I_2 + w_n} \otimes  \langle g, h_{I_2 + w_n}\rangle_1 \rangle.
\end{align*}
Splitting the trivial representation in to these two parts allows us to conclude that
\begin{displaymath}
\langle Tf, g\rangle = \frac{1}{\pi_{\textrm{good}}^n}E_{w_n}  \sum_{I_1, I_2 \in \mathcal{D}_n} \chi_{\textrm{good}}(\textrm{smaller}(I_1,I_2))  \langle T(h_{I_1} \otimes \langle f, h_{I_1}\rangle_1), h_{I_2} \otimes  \langle g, h_{I_2}\rangle_1 \rangle.
\end{displaymath}

We now expand on $\R^m$. One may write
\begin{displaymath}
\langle f, h_{I_1}\rangle_1 = \sum_{J_1 \in \mathcal{D}_m} \langle f, h_{I_1} \otimes u_{J_1}\rangle u_{J_1}
\end{displaymath}
so that
\begin{displaymath}
h_{I_1} \otimes \langle f, h_{I_1}\rangle_1 = \sum_{J_1 \in \mathcal{D}_m} \langle f, h_{I_1} \otimes u_{J_1}\rangle h_{I_1} \otimes u_{J_1}.
\end{displaymath}
We may then follow the recipe from above: insert this to the above formula for $\langle Tf, g\rangle$, add goodness to $J_1$ by independence, expand $h_{I_2} \otimes  \langle g, h_{I_2}\rangle_1$, split the summation to $\ell(J_1) \le \ell(J_2)$ and
$\ell(J_1) > \ell(J_2)$, remove the goodness from $J_1$ in the latter summation by independence and, finally, compare to the appropriate trivial identity. One also does the symmetric thing, where one first expands $h_{I_2} \otimes  \langle g, h_{I_2}\rangle_1$
and adds the goodness to $J_2$. Combining these gives the claim of the proposition.
\end{proof}
\begin{rem}
One may also use full expansions like $f = \sum_{I_1 \in \mathcal{D}_n} \sum_{J_1 \in \mathcal{D}_m} \langle f, h_{I_1} \otimes u_{J_1} \rangle h_{I_1} \otimes u_{J_1}$ in the beginning of the proof. Following the usual trickery
this leads to the formula
\begin{align*}
\langle Tf, g\rangle = \frac{1}{\pi_{\textrm{good}}^n}  \mathbb{E} \sum_{I_1, I_2 \in \mathcal{D}_n} \sum_{J_1, J_2 \in \mathcal{D}_m}
&\chi_{\textrm{good}}(\textrm{smaller}(I_1,I_2)) \\
& \langle T(h_{I_1} \otimes u_{J_1}), h_{I_2} \otimes u_{J_2}\rangle \langle f, h_{I_1} \otimes u_{J_1}\rangle \langle g, h_{I_2} \otimes u_{J_2}\rangle.
\end{align*}
Here it may at first seem that there is no longer enough independence to add the goodness to $J_1$. However, one may simply write the summation as
\begin{align*}
 \sum_{I_1, I_2 \in \mathcal{D}_n} \sum_{J_1\in \mathcal{D}_m}  \chi_{\textrm{good}}(\textrm{smaller}(I_1,I_2))  \langle T(h_{I_1} \otimes u_{J_1}), g_{I_2}\rangle \langle f, h_{I_1} \otimes u_{J_1}\rangle,
\end{align*}
where one realizes that
\begin{displaymath}
g_{I_2} = \sum_{J_2 \in \mathcal{D}_m}\langle g, h_{I_2} \otimes u_{J_2}\rangle h_{I_2} \otimes u_{J_2} = h_{I_2} \otimes \langle g, h_{I_2} \rangle_1
\end{displaymath}
does not depend on $w_m$. Then one may add the goodness to $J_1$ using independence and repeat the basic recipe to get the proposition.
\end{rem}

\subsection*{Strategy and formulation of the main theorem}
We fix the random variables $w_n$ and $w_m$ which fixes the dyadic grids $\mathcal{D}_n$ and $\mathcal{D}_m$ respectively. Then we study the summation
\begin{displaymath}
\mathop{\sum_{\ell(I_1) \le \ell(I_2)}}_{I_1 \textrm{ good}} \mathop{\sum_{\ell(J_1) \le \ell(J_2)}}_{J_1 \textrm{ good}}
\langle T(h_{I_1} \otimes u_{J_1}), h_{I_2} \otimes u_{J_2}\rangle \langle f, h_{I_1} \otimes u_{J_1}\rangle \langle g, h_{I_2} \otimes u_{J_2}\rangle.
\end{displaymath}
We more often than not suppress from the notation the important fact that $I_1$ and $J_1$ are good. Then we perform the splitting
\begin{displaymath}
\sum_{\ell(I_1) \le \ell(I_2)} = \mathop{\sum_{\ell(I_1) \le \ell(I_2)}}_{d(I_1,I_2) > \ell(I_1)^{\gamma_n}\ell(I_2)^{1-\gamma_n}} + \sum_{I_1 \subsetneq I_2} +  \sum_{I_1 = I_2}
+  \mathop{\mathop{\sum_{\ell(I_1) \le \ell(I_2)}}_{d(I_1,I_2) \le \ell(I_1)^{\gamma_n}\ell(I_2)^{1-\gamma_n}}}_{I_1 \cap I_2 = \emptyset},
\end{displaymath}
and similarly for the summation over the grid $\mathcal{D}_m$. Here $d(A,B)$ denotes the distance of the sets $A$ and $B$ (recall that we use the $\ell^{\infty}$ metric).
The first sum is the separated sum, then we have the inside sum, the equal sum and the nearby sum.
The summation over both the grids is split in to various types which also includes several mixed types. The list is:
separated/separated, separated/inside, separated/equal, separated/nearby, inside/inside, inside/equal, inside/nearby, equal/equal, equal/nearby,
nearby/nearby and some symmetric mixed sums. It seems reasonable to deal with these separately.

Note that actually the mixed sums where $\ell(I_1) \le \ell(I_2)$ and $\ell(J_1) > \ell(J_2)$ or $\ell(I_1) > \ell(I_2)$ and $\ell(J_1) \le \ell(J_2)$ are not completely symmetrical to this case.
However, the relevant difference is only in the full paraproduct that appears in the corresponding inside/inside part. There one gets a bit different paraproducts, which are related to
the assumptions that $T_1(1)$ and $T_1^*(1)$ belong to the product BMO of $\R^n \times \R^m$. 
We comment more on this on Remark \ref{different par}.

The goal is to represent all of these different parts as a sum of shifts with a good decay factor in front. Combining all these cases together leads to our main theorem:
\begin{thm}
For a bi-parameter singular integral operator $T$ as defined above, there holds for some bi-parameter shifts $S^{i_1i_2j_1j_2}_{\mathcal{D}_n\mathcal{D}_m}$ that
\begin{displaymath}
\langle Tf, g\rangle = C_T \mathbb{E}_{w_n} \mathbb{E}_{w_m}\mathop{\sum_{(i_1, i_2) \in \Z_+^2}}_{(j_1, j_2) \in \Z_+^2} 2^{-\max(i_1,i_2)\delta/2}2^{-\max(j_1,j_2)\delta/2} \langle S^{i_1i_2j_1j_2}_{\mathcal{D}_n\mathcal{D}_m}f, g \rangle,
\end{displaymath}
where non-cancellative shifts may only appear if $(i_1, i_2) = (0,0)$ or $(j_1, j_2) = (0,0)$.
\end{thm}
\begin{cor}
A bi-parameter singular integral $T$ as defined above is $L^2$ bounded.
\end{cor}

We note that all of the appearing non-cancellative shifts will have a certain paraproduct structure, and this structure is explicit in the proof. For example in \cite{HLMORSUT}, where the one-parameter representation theorem is applied,
it is important to know the explicit structure of the non-cancellative shifts.

The rest of the paper is dedicated to the piece by piece proof of this theorem. We use $X \lesssim Y$ to mean $X \le CY$ for some constant $C$ and $X \sim Y$ to mean $Y \lesssim X \lesssim Y$. Of course, we cannot absorb just any constants, but
only ones that depend on the dimensions or the various constants from the assumptions concerning $T$.
\section{Separated/separated}
Let $I_1 \vee I_2 = \bigcap_{K \in \mathcal{D}_n, \, K \supset I_1 \cup I_2} K$ and $J_1 \vee J_2 = \bigcap_{V \in \mathcal{D}_m, \, V \supset J_1 \cup J_2} V$.
The separation conditions together with goodness imply $\ell(I_1)^{\gamma_n}\ell(I_1 \vee I_2)^{1-\gamma_n} \lesssim d(I_1,I_2)$ and $\ell(J_1)^{\gamma_m}\ell(J_1 \vee J_2)^{1-\gamma_m} \lesssim d(J_1,J_2)$.

Let us write
\begin{align*}
&\mathop{\sum_{\ell(I_1) \le \ell(I_2)}}_{d(I_1,I_2) > \ell(I_1)^{\gamma_n}\ell(I_2)^{1-\gamma_n}} \mathop{\sum_{\ell(J_1) \le \ell(J_2)}}_{d(J_1,J_2) > \ell(J_1)^{\gamma_m}\ell(J_2)^{1-\gamma_m}} \\
&= \mathop{\sum_{i_2 \ge 1}}_{j_2 \ge 1}  \mathop{\sum_{i_1 \ge i_2}}_{j_1 \ge j_2} \mathop{\sum_{K \in \mathcal{D}_n}}_{V \in \mathcal{D}_m}  
\mathop{\mathop{\sum_{d(I_1,I_2) > \ell(I_1)^{\gamma_n}\ell(I_2)^{1-\gamma_n}}}_{I_1 \vee I_2 = K}}_{\ell(I_1) = 2^{-i_1}\ell(K), \,\ell(I_2) = 2^{-i_2}\ell(K)}
\mathop{\mathop{\sum_{d(J_1,J_2) > \ell(J_1)^{\gamma_m}\ell(J_2)^{1-\gamma_m}}}_{J_1 \vee J_2 = V}}_{\ell(J_1) = 2^{-j_1}\ell(V), \,\ell(J_2) = 2^{-j_2}\ell(V)}.
\end{align*}
\begin{lem}
For $I_1,I_2, J_1, J_2$ in the above summation, we have the estimate
\begin{align*}
|\langle T(h_{I_1} \otimes u_{J_1}), &h_{I_2} \otimes u_{J_2}\rangle| \\
&\lesssim \frac{|I_1|^{1/2}|I_2|^{1/2}}{|K|} \frac{|J_1|^{1/2}|J_2|^{1/2}}{|V|} \Big(\frac{\ell(I_1)}{\ell(K)}\Big)^{\delta/2}  \Big(\frac{\ell(J_1)}{\ell(V)}\Big)^{\delta/2} \\
&=  2^{-i_1\delta/2}\frac{|I_1|^{1/2}|I_2|^{1/2}}{|K|} 2^{-j_1\delta/2} \frac{|J_1|^{1/2}|J_2|^{1/2}}{|V|}. 
\end{align*}
\end{lem}
\begin{proof}
Given a cube $I$ we denote by $c_I$ its center.
We may write
\begin{align*}
&\langle T(h_{I_1} \otimes u_{J_1}), h_{I_2} \otimes u_{J_2}\rangle \\
&= \int_{I_1 \times J_1} \int_{I_2 \times J_2} K(x,y) h_{I_1}(y_1)u_{J_1}(y_2)h_{I_2}(x_1)u_{J_2}(x_2)\,dxdy,
\end{align*}
where we may, using cancellation, replace $K(x,y)$ by
\begin{displaymath}
K(x,y) - K(x, (y_1, c_{J_1})) - K(x, (c_{I_1}, y_2)) + K(x, (c_{I_1}, c_{J_1})).
\end{displaymath}
Since $|y_1 - c_{I_{1}}| \le \ell(I_1)/2 \le \frac{1}{2}\ell(I_1)^{\gamma_n} \ell(I_2)^{1-\gamma_n} \le d(I_1,I_2)/2 \le |x_1 - c_{I_1}|/2$
and similarly $|y_2-c_{J_1}| \le |x_2 - c_{J_1}|/2$, we have
\begin{align*}
|K(x,y) - &K(x, (y_1, c_{J_1})) - K(x, (c_{I_1}, y_2)) + K(x, (c_{I_1}, c_{J_1}))|  \\
&\lesssim \frac{|y_1 - c_{I_{1}}|^{\delta}}{ |x_1 - c_{I_1}|^{n+\delta}}  \frac{|y_2 - c_{J_{1}}|^{\delta}}{ |x_2 - c_{J_1}|^{m+\delta}} \\
&\lesssim \ell(I_1)^{\delta}d(I_1,I_2)^{-n-\delta} \ell(J_1)^{\delta} d(J_1,J_2)^{-m-\delta} \\
&\lesssim \ell(I_1)^{\delta}[\ell(I_1)^{\gamma_n}\ell(K)^{1-\gamma_n}]^{-n-\delta} \ell(J_1)^{\delta} [\ell(J_1)^{\gamma_m}\ell(V)^{1-\gamma_m}]^{-m-\delta} \\
&= \ell(I_1)^{\delta/2}\ell(K)^{-\delta/2}|K|^{-1} \ell(J_1)^{\delta/2}\ell(V)^{-\delta/2}|V|^{-1}.
\end{align*}
Here we used $\ell(I_1)^{\gamma_n}\ell(K)^{1-\gamma_n} \lesssim d(I_1,I_2)$ and $\gamma_n n + \gamma_n \delta = \delta/2$ (and the analogous estimates involving $J_1$, $J_2$, $V$ and $m$).
Recalling the $L^2$ normalization of the Haar functions and the fact that $\ell(I_1)/\ell(K) = 2^{-i_1}$ and $\ell(J_1)/\ell(V) = 2^{-j_1}$
completes the proof.
\end{proof}
We write
\begin{align*}
\langle T(h_{I_1} \otimes u_{J_1}),& h_{I_2} \otimes u_{J_2}\rangle \langle f, h_{I_1} \otimes u_{J_1}\rangle \langle g, h_{I_2} \otimes u_{J_2}\rangle \\
&= C2^{-i_1\delta/2}2^{-j_1\delta/2} \frac{\langle T(h_{I_1} \otimes u_{J_1}), h_{I_2} \otimes u_{J_2}\rangle}{C2^{-i_1\delta/2}2^{-j_1\delta/2} } \langle \langle f, h_{I_1} \otimes u_{J_1} \rangle h_{I_2} \otimes u_{J_2}, g\rangle.
\end{align*}
Define
\begin{displaymath}
a_{I_1I_2KJ_1J_2V} = \frac{\langle T(h_{I_1} \otimes u_{J_1}), h_{I_2} \otimes u_{J_2}\rangle}{C2^{-i_1\delta/2}2^{-j_1\delta/2} }
\end{displaymath}
if all the various goodness and separation conditions appearing in the summations are satisfied, and otherwise set $a_{I_1I_2KJ_1J_2V}  = 0$.
This enables us to write
\begin{align*}
\mathop{\sum_{\ell(I_1) \le \ell(I_2)}}_{d(I_1,I_2) > \ell(I_1)^{\gamma_n}\ell(I_2)^{1-\gamma_n}} \mathop{\sum_{\ell(J_1) \le \ell(J_2)}}_{d(J_1,J_2) > \ell(J_1)^{\gamma_m}\ell(J_2)^{1-\gamma_m}}
\langle T(h_{I_1}& \otimes u_{J_1}), h_{I_2} \otimes u_{J_2} \rangle \\ &\langle f, h_{I_1} \otimes u_{J_1}\rangle \langle g, h_{I_2} \otimes u_{J_2}\rangle 
\end{align*}
in the form
\begin{displaymath}
C\mathop{\sum_{i_2 \ge 1}}_{j_2 \ge 1}  \mathop{\sum_{i_1 \ge i_2}}_{j_1 \ge j_2} 2^{-i_1\delta/2}2^{-j_1\delta/2} \sum_{K,\,V} \langle A^{i_1i_2j_1j_2}_{KV}f, g\rangle,
\end{displaymath}
where
\begin{displaymath}
A^{i_1i_2j_1j_2}_{KV}f = \mathop{\mathop{\sum_{I_1,\,I_2 \subset K}}_{\ell(I_1) = 2^{-i_1}\ell(K)}}_{\ell(I_2) = 2^{-i_2}\ell(K)}
\mathop{\mathop{\sum_{J_1, \,J_2 \subset V}}_{\ell(J_1) = 2^{-j_1}\ell(V)}}_{\ell(J_2) = 2^{-j_2}\ell(V)}
 a_{I_1I_2KJ_1J_2V} \langle f, h_{I_1} \otimes u_{J_1} \rangle h_{I_2} \otimes u_{J_2}
\end{displaymath}
with
\begin{displaymath}
| a_{I_1I_2KJ_1J_2V}| \le \frac{|I_1|^{1/2}|I_2|^{1/2}}{|K|}\frac{|J_1|^{1/2}|J_2|^{1/2}}{|V|}.
\end{displaymath}
The corresponding bi-parameter shift with indices $i_1, i_2, j_1, j_2$ is by definition
\begin{displaymath}
S^{i_1i_2j_1j_2}f = \sum_{K,V} A^{i_1i_2j_1j_2}_{KV}f.
\end{displaymath}
\section{Separated/inside}
As $J_1 \subsetneq J_2$, there is a child $J_{2,1}$ of $J_2$ such that $J_1 \subset J_{2,1}$. We decompose
\begin{align*}
\langle T(h_{I_1} \otimes u_{J_1}), h_{I_2} \otimes u_{J_2}\rangle = \langle T(h_{I_1}& \otimes u_{J_1}), h_{I_2} \otimes s_{J_1J2}\rangle \\
&+ \langle u_{J_2}\rangle_{J_1} \langle T(h_{I_1} \otimes u_{J_1}), h_{I_2} \otimes 1\rangle,
\end{align*}   
where $s_{J_1J_2} = \chi_{J_{2,1}^c}[u_{J_2} - \langle u_{J_2} \rangle_{J_{2,1}}]$. The relevant properties of $s_{J_1J_2}$ are $|s_{J_1J_2}| \le 2|J_2|^{-1/2}$ and spt$\,s_{J_1J_2} \subset J_{2,1}^c$.

We write
\begin{align*}
&\mathop{\sum_{\ell(I_1) \le \ell(I_2)}}_{d(I_1,I_2) > \ell(I_1)^{\gamma_n}\ell(I_2)^{1-\gamma_n}}  \sum_{J_1 \subsetneq J_2} \\& =  \sum_{i_2 \ge 1} \sum_{i_1 \ge i_2} \sum_{j_1 \ge 1}
\sum_{K \in \mathcal{D}_n} \sum_{J_2 \in \mathcal{D}_m} \mathop{\mathop{\sum_{d(I_1,I_2) > \ell(I_1)^{\gamma_n}\ell(I_2)^{1-\gamma_n}}}_{I_1 \vee I_2 = K}}_{\ell(I_1) = 2^{-i_1}\ell(K), \,\ell(I_2) = 2^{-i_2}\ell(K)}
\mathop{\sum_{J_1 \subset J_2}}_{\ell(J_1) = 2^{-j_1}\ell(J_2)}.
\end{align*}

\begin{lem}\label{sep,in}
For $I_1,I_2, J_1, J_2$ in the above summation, we have the estimate
\begin{align*}
|\langle T(h_{I_1} \otimes u_{J_1})&, h_{I_2} \otimes s_{J_1J2}\rangle| \\
&\lesssim \frac{|I_1|^{1/2}|I_2|^{1/2}}{|K|} \frac{|J_1|^{1/2}}{|J_2|^{1/2}} \Big(\frac{\ell(I_1)}{\ell(K)}\Big)^{\delta/2}  \Big(\frac{\ell(J_1)}{\ell(J_2)}\Big)^{\delta/2} \\
&=  2^{-i_1\delta/2}\frac{|I_1|^{1/2}|I_2|^{1/2}}{|K|} 2^{-j_1\delta/2} \frac{|J_1|^{1/2}}{|J_2|^{1/2}}. 
\end{align*}
\end{lem}
\begin{proof}
There is good separation by the goodness of $J_1$ if $\ell(J_1) < 2^{-r}\ell(J_2)$. Indeed, in this case there holds $d(J_1, J_{2,1}^c) \ge
2\ell(J_1)^{\gamma_m}\ell(J_{2,1})^{1-\gamma_m} \ge \ell(J_1)^{\gamma_m}\ell(J_2)^{1-\gamma_m}$. Then we may write
\begin{align*}
\langle T(h_{I_1} \otimes u_{J_1}), &h_{I_2} \otimes s_{J_1J2}\rangle \\ &= \int_{I_1 \times J_1} \int_{I_2 \times J_{2,1}^c}  K(x,y) h_{I_1}(y_1)u_{J_1}(y_2)h_{I_2}(x_1) s_{J_1J_2}(x_2)\,dx\,dy,
\end{align*}
and replace $K(x,y)$ by $K(x,y) - K(x, (y_1, c_{J_1})) - K(x, (c_{I_1}, y_2)) + K(x, (c_{I_1}, c_{J_1}))$ using the cancellation of $u_{J_1}$ and $h_{I_1}$. We may utilize the kernel estimates to get
\begin{align*}
|K(x,y) - K(x, (y_1, c_{J_1})) -& K(x, (c_{I_1}, y_2)) + K(x, (c_{I_1}, c_{J_1}))| \\
&\lesssim \ell(I_1)^{\delta/2}\ell(K)^{-\delta/2}|K|^{-1} \ell(J_1)^{\delta} \frac{1}{|x_2-c_{J_1}|^{m+\delta}}.
\end{align*}
This yields
\begin{align*}
|\langle T(h_{I_1} \otimes u_{J_1})&, h_{I_2} \otimes s_{J_1J2}\rangle| \\ 
&\lesssim \frac{|I_1|^{1/2}|I_2|^{1/2}}{|K|}  \Big(\frac{\ell(I_1)}{\ell(K)}\Big)^{\delta/2}  \frac{|J_1|^{1/2}}{|J_2|^{1/2}} \ell(J_1)^{\delta} \int_{J_{2,1}^c} \frac{dx_2}{|x_2-c_{J_1}|^{m+\delta}},
\end{align*}
where
\begin{align*}
\int_{J_{2,1}^c} \frac{dx_2}{|x_2-c_{J_1}|^{m+\delta}} &\lesssim \int_{\R^m \setminus B(c_{J_1}, d(J_1,J_{2,1}^c))} \frac{dx_2}{|x_2-c_{J_1}|^{m+\delta}} \\
&\lesssim d(J_1,J_{2,1}^c)^{-\delta} \lesssim  \ell(J_1)^{-\delta/2}\ell(J_2)^{-\delta/2}.
\end{align*}
Therefore, we have
\begin{align*}
|\langle T(h_{I_1} \otimes u_{J_1})&, h_{I_2} \otimes s_{J_1J2}\rangle| \\
&\lesssim \frac{|I_1|^{1/2}|I_2|^{1/2}}{|K|}  \Big(\frac{\ell(I_1)}{\ell(K)}\Big)^{\delta/2}  \frac{|J_1|^{1/2}}{|J_2|^{1/2}}  \Big(\frac{\ell(J_1)}{\ell(J_2)}\Big)^{\delta/2}.
\end{align*}

We still need to deal with the case $2^{-r}\ell(J_2) \le \ell(J_1) (\le \ell(J_2))$. This time we split
\begin{align*}
\langle T(h_{I_1} \otimes u_{J_1}), h_{I_2} \otimes s_{J_1J2}\rangle = \langle T(h_{I_1} &\otimes u_{J_1}), h_{I_2} \otimes (\chi_{3J_1}s_{J_1J2})\rangle \\ &+ \langle T(h_{I_1} \otimes u_{J_1}), h_{I_2} \otimes (\chi_{(3J_1)^c}s_{J_1J2})\rangle.
\end{align*}
We have that $\langle T(h_{I_1} \otimes u_{J_1}), h_{I_2} \otimes (\chi_{3J_1}s_{J_1J2})\rangle$ equals
\begin{displaymath}
\int_{I_1 \times J_1} \int_{I_2 \times (3J_1 \setminus J_{2,1})} [K(x,y) - K(x,(c_{I_1}, y_2))]h_{I_1}(y_1)u_{J_1}(y_2)h_{I_2}(x_1)s_{J_1J_2}(x_2)\,dx\,dy
\end{displaymath}
so we can estimate using the mixed H\"older and size estimate that
\begin{align*}
&|\langle T(h_{I_1} \otimes u_{J_1}), h_{I_2} \otimes (\chi_{3J_1}s_{J_1J2})\rangle| \\
&\lesssim \frac{|I_1|^{1/2}|I_2|^{1/2}}{|K|} \Big(\frac{\ell(I_1)}{\ell(K)}\Big)^{\delta/2} |J_1|^{-1/2}|J_2|^{-1/2} \int_{J_1} \int_{3J_1 \setminus J_1} \frac{1}{|x_2-y_2|^m}\,dx_2\,dy_2 \\
&\lesssim  \frac{|I_1|^{1/2}|I_2|^{1/2}}{|K|} \Big(\frac{\ell(I_1)}{\ell(K)}\Big)^{\delta/2}  \frac{|J_1|^{1/2}}{|J_2|^{1/2}} \\
&\lesssim \frac{|I_1|^{1/2}|I_2|^{1/2}}{|K|} \Big(\frac{\ell(I_1)}{\ell(K)}\Big)^{\delta/2}  \frac{|J_1|^{1/2}}{|J_2|^{1/2}} \Big(\frac{\ell(J_1)}{\ell(J_2)}\Big)^{\delta/2}.
\end{align*}

In the term $\langle T(h_{I_1} \otimes u_{J_1}), h_{I_2} \otimes (\chi_{(3J_1)^c}s_{J_1J2})\rangle$ we have good separation everywhere, so the H\"older estimate for $K$ yields
\begin{align*}
|\langle T(h_{I_1} \otimes & u_{J_1}), h_{I_2} \otimes (\chi_{(3J_1)^c}s_{J_1J2})\rangle| \\
&\lesssim \frac{|I_1|^{1/2}|I_2|^{1/2}}{|K|}  \Big(\frac{\ell(I_1)}{\ell(K)}\Big)^{\delta/2}  \frac{|J_1|^{1/2}}{|J_2|^{1/2}} \ell(J_1)^{\delta} \int_{(3J_1)^c} \frac{dx_2}{|x_2-c_{J_1}|^{m+\delta}} \\
&\lesssim \frac{|I_1|^{1/2}|I_2|^{1/2}}{|K|}  \Big(\frac{\ell(I_1)}{\ell(K)}\Big)^{\delta/2}  \frac{|J_1|^{1/2}}{|J_2|^{1/2}} \\
&\lesssim \frac{|I_1|^{1/2}|I_2|^{1/2}}{|K|}  \Big(\frac{\ell(I_1)}{\ell(K)}\Big)^{\delta/2}  \frac{|J_1|^{1/2}}{|J_2|^{1/2}}  \Big(\frac{\ell(J_1)}{\ell(J_2)}\Big)^{\delta/2}.
\end{align*}
\end{proof}
The above lemma enables us to write
\begin{displaymath}
\mathop{\sum_{\ell(I_1) \le \ell(I_2)}}_{d(I_1,I_2) > \ell(I_1)^{\gamma_n}\ell(I_2)^{1-\gamma_n}}  \sum_{J_1 \subsetneq J_2} \langle T(h_{I_1} \otimes u_{J_1}), h_{I_2} \otimes s_{J_1J2}\rangle
\langle f, h_{I_1} \otimes u_{J_1}\rangle \langle g, h_{I_2} \otimes u_{J_2}\rangle
\end{displaymath}
in the form
\begin{displaymath}
C\sum_{i_2 \ge 1} \sum_{i_1 \ge i_2} \sum_{j_1 \ge 1} 2^{-i_1\delta/2} 2^{-j_1\delta/2} \langle S^{i_1i_2j_10}f,g\rangle.
\end{displaymath}

Next, we deal with the series with the term $\langle u_{J_2}\rangle_{J_1} \langle T(h_{I_1} \otimes u_{J_1}), h_{I_2} \otimes 1\rangle$. This will yield shifts of the type $(i_1,i_2, 0,0)$ which are non-cancellative (their $\R^m$ parts are paraproducts in a certain sense).
As these shifts will be non-cancellative, we will also have to worry about their $L^2$ boundedness properties.

Write
\begin{align*}
&\sum_{J_1 \subsetneq J_2} \langle u_{J_2}\rangle_{J_1} \langle T(h_{I_1} \otimes u_{J_1}), h_{I_2} \otimes 1\rangle \langle f, h_{I_1} \otimes u_{J_1}\rangle  \langle g, h_{I_2} \otimes u_{J_2}\rangle \\
&= \sum_{J_1} \Big\langle  \sum_{J_2} \langle g, h_{I_2} \otimes u_{J_2}\rangle u_{J_2}\Big\rangle_{J_1} \langle T(h_{I_1} \otimes u_{J_1}), h_{I_2} \otimes 1\rangle \langle f, h_{I_1} \otimes u_{J_1}\rangle \\
&= \sum_V \langle \langle g, h_{I_2}\rangle_1 \rangle_{V}   \langle T(h_{I_1} \otimes u_{V}), h_{I_2} \otimes 1\rangle \langle f, h_{I_1} \otimes u_{V}\rangle.
\end{align*}
The summands can further be written in the form
\begin{displaymath}
|V|^{-1/2}  \langle T(h_{I_1} \otimes u_{V}), h_{I_2} \otimes 1\rangle \langle \langle f, h_{I_1} \otimes u_{V}\rangle h_{I_2} \otimes u^0_V, g\rangle,
\end{displaymath}
where $u^0_V = |V|^{-1/2}\chi_V$. Written in this way it is evident that we will have the required shift structure of the type $(i_1, i_2, 0, 0)$.
\begin{lem}\label{normali}
The correct normalization
\begin{displaymath}
|\langle T(h_{I_1} \otimes u_{V}), h_{I_2} \otimes 1\rangle| \lesssim \frac{|I_1|^{1/2}|I_2|^{1/2}}{|K|}\Big(\frac{\ell(I_1)}{\ell(K)}\Big)^{\delta/2}|V|^{1/2}
\end{displaymath}
holds.
\end{lem}
\begin{proof}
Let us first split
\begin{displaymath}
\langle T(h_{I_1} \otimes u_{V}), h_{I_2} \otimes 1\rangle =  \langle T(h_{I_1} \otimes u_{V}), h_{I_2} \otimes  \chi_{3V} \rangle + \langle T(h_{I_1} \otimes u_{V}), h_{I_2} \otimes \chi_{(3V)^c}\rangle.
\end{displaymath}
We have
\begin{align*}
|\langle T(h_{I_1} \otimes u_{V}), h_{I_2} \otimes \chi_{3V} \rangle| \le |V|^{-1/2} \sum_{V' \in \textup{ch}(V)} \Big[ |\langle T(h_{I_1} &\otimes \chi_{V'}), h_{I_2} \otimes \chi_{3V \setminus V'} \rangle| \\
&+  |\langle T(h_{I_1} \otimes \chi_{V'}), h_{I_2} \otimes \chi_{V'} \rangle|\Big]
\end{align*}
For the first time, we use the kernel representations in $\R^n$ to write $\langle T(h_{I_1} \otimes \chi_{V'}), h_{I_2} \otimes \chi_{V'} \rangle$ in the form
\begin{displaymath}
\int_{I_1} \int_{I_2} [K_{\chi_{V'}, \chi_{V'}}(x_1, y_1)-K_{\chi_{V'}, \chi_{V'}}(x_1, c_{I_1})]   h_{I_1}(y_1)h_{I_2}(x_1)\,dx_1\,dy_1.
\end{displaymath}
This gives that
\begin{align*}
|\langle T(h_{I_1} \otimes \chi_{V'}), h_{I_2} \otimes \chi_{V'} \rangle| &\le C(\chi_{V'}, \chi_{V'}) \frac{|I_1|^{1/2}|I_2|^{1/2}}{|K|}\Big(\frac{\ell(I_1)}{\ell(K)}\Big)^{\delta/2} \\
&\lesssim |V|\frac{|I_1|^{1/2}|I_2|^{1/2}}{|K|}\Big(\frac{\ell(I_1)}{\ell(K)}\Big)^{\delta/2}.
\end{align*}
Notice that by the mixed H\"older and size estimates for $K$ we have the same bound also for the term $ |\langle T(h_{I_1} \otimes \chi_{V'}), h_{I_2} \otimes \chi_{3V \setminus V'} \rangle|$, and
so there holds
\begin{displaymath}
|\langle T(h_{I_1} \otimes u_{V}), h_{I_2} \otimes \chi_{3V} \rangle| \lesssim \frac{|I_1|^{1/2}|I_2|^{1/2}}{|K|}\Big(\frac{\ell(I_1)}{\ell(K)}\Big)^{\delta/2} |V|^{1/2}.
\end{displaymath}
The term  $\langle T(h_{I_1} \otimes u_{V}), h_{I_2} \otimes \chi_{(3V)^c}\rangle$ is in control by the full kernel representation and the H\"older estimate for $K$.
\end{proof}
These are non-cancellative shifts so we must separately demonstrate the $L^2$ boundedness. For this, we prefer
to write things in a different way:
\begin{align*}
&\sum_V \langle \langle g, h_{I_2}\rangle_1 \rangle_{V}   \langle T(h_{I_1} \otimes u_{V}), h_{I_2} \otimes 1\rangle \langle f, h_{I_1} \otimes u_{V}\rangle \\
&= \sum_V \langle \langle g, h_{I_2}\rangle_1 \rangle_{V} \langle \langle T^*(h_{I_2} \otimes 1), h_{I_1}\rangle_1, u_V\rangle \langle \langle f, h_{I_1} \rangle_1, u_V\rangle \\
&= C2^{-i_1\delta/2}\Big\langle \langle f, h_{I_1} \rangle_1, \sum_V \langle \langle g, h_{I_2}\rangle_1 \rangle_{V} \langle b_{I_1I_2}, u_V\rangle u_V \Big\rangle \\
&= C2^{-i_1\delta/2}\langle \langle f, h_{I_1} \rangle_1, \Pi_{b_{I_1I_2}}(\langle g, h_{I_2}\rangle_1) \rangle \\
&= C2^{-i_1\delta/2}\langle \Pi_{b_{I_1I_2}}^*(\langle f, h_{I_1} \rangle_1), \langle g, h_{I_2}\rangle_1 \rangle \\
&= C2^{-i_1\delta/2}\langle h_{I_2} \otimes \Pi_{b_{I_1I_2}}^*(\langle f, h_{I_1} \rangle_1), g\rangle,
\end{align*}
where $b_{I_1I_2} = \langle T^*(h_{I_2} \otimes 1), h_{I_1}\rangle_1 /(C2^{-i_1\delta/2})$ and $\Pi_{b_{I_1I_2}}$ is the related paraproduct on $\R^m$ defined by the general formula
\begin{displaymath}
\Pi_{b}a = \sum_V \langle a \rangle_V \langle b, u_V\rangle u_V.
\end{displaymath}
\begin{lem}\label{firstbmo}
We have $b_{I_1I_2} \in \textup{BMO}(\R^m)$ with the bound
\begin{displaymath}
\|b_{I_1I_2}\|_{ \textup{BMO}(\R^m)} \le c\frac{|I_1|^{1/2}|I_2|^{1/2}}{|K|}.
\end{displaymath}
\end{lem}
\begin{proof}
Let $V$ be any cube in $\R^m$ and $a$ be any function in $\R^m$ such that spt$\,a \subset V$, $|a| \le 1$ and $\int a = 0$. It suffices to show that
\begin{displaymath}
|\langle T(h_{I_1} \otimes a), h_{I_2} \otimes 1\rangle| \lesssim \frac{|I_1|^{1/2}|I_2|^{1/2}}{|K|} \Big(\frac{\ell(I_1)}{\ell(K)}\Big)^{\delta/2}|V|.
\end{displaymath}
This is done by splitting $1 = \chi_{3V} + \chi_{(3V)^c}$ and using kernel estimates in a similar fashion as before.
\end{proof}
\begin{rem}
The strengthening of  Lemma \ref{normali} to the related BMO estimate of Lemma \ref{firstbmo}
requires one to have the control $C(u_V, \chi_V) \le C|V|$ for $V$-adapted functions $u_V$ with zero-mean. It is precisely for these type of BMO reasons that merely the assumption $C(\chi_V, \chi_V) \le C|V|$ does not seem to be enough
for the results of this paper.
\end{rem}
Let us abbreviate
\begin{displaymath}
\mathop{\sum_{I_1,\, I_2 \subset K}}_{\ell(I_1) = 2^{-i_1}\ell(K),\, \ell(I_2) = 2^{-i_2}\ell(K)} = \sum_{I_1,\, I_2 \subset K}^{(i_1,i_2)}.
\end{displaymath}
We are ready to show the boundedness of our non-cancellative shifts of type $(i_1,i_2,0,0)$.
\begin{prop}\label{singlepara}
There holds
\begin{displaymath}
\Big\|\sum_{K} \sum_{I_1,\, I_2 \subset K}^{(i_1,i_2)}  h_{I_2} \otimes \Pi_{b_{I_1I_2}}^*(\langle f, h_{I_1} \rangle_1)\Big\|_2 \le \|f\|_2.
\end{displaymath}
\end{prop}
\begin{proof}
There holds by orthogonality that
\begin{align*}
\Big\|\sum_{K} \sum_{I_1,\, I_2 \subset K}^{(i_1,i_2)}  h_{I_2}& \otimes \Pi_{b_{I_1I_2}}^*(\langle f, h_{I_1} \rangle_1)\Big\|_2^2 \\
&= \sum_K \sum_{I_2 \subset K}^{(i_2)} \Big\| \sum_{I_1 \subset K}^{(i_1)}   \Pi_{b_{I_1I_2}}^*(\langle f, h_{I_1} \rangle_1)\Big\|_2^2 \\
&\le \sum_K \sum_{I_2 \subset K}^{(i_2)} \Big( \sum_{I_1 \subset K}^{(i_1)} \| \Pi_{b_{I_1I_2}}^*(\langle f, h_{I_1} \rangle_1)\|_2 \Big)^2.
\end{align*}
Let $p^{i_1}_K$ be the orthogonal projection from $L^2(\R^n)$ to span$\{h_{I_1}: \, I_1 \subset K,\, \ell(I_1) = 2^{-i_1}\ell(K)\}$. Write also
$f_y(x) = f(x,y)$. There holds by the boundedness of paraproducts defined by BMO functions and the previous lemma that
\begin{align*}
\| \Pi_{b_{I_1I_2}}^*(\langle f, h_{I_1} \rangle_1 \|_2 &\le \frac{|I_1|^{1/2}|I_2|^{1/2}}{|K|} \|\langle f, h_{I_1} \rangle_1\|_2 \\
&\le \frac{|I_1|^{1/2}|I_2|^{1/2}}{|K|} \Big( \int_{\R^m} \int_{I_1} |p^{i_1}_K f_y(x)|^2\,dx\,dy\Big)^{1/2}.
\end{align*}
Therefore, we have
\begin{align*}
&\Big\|\sum_{K} \sum_{I_1,\, I_2 \subset K}^{(i_1,i_2)}  h_{I_2} \otimes \Pi_{b_{I_1I_2}}^*(\langle f, h_{I_1} \rangle_1)\Big\|_2^2 \\
&\le \sum_K \frac{1}{|K|} \Big( \sum_{I_1 \subset K}^{(i_1)} |I_1|^{1/2}    \Big( \int_{\R^m} \int_{I_1} |p^{i_1}_K f_y(x)|^2\,dx\,dy\Big)^{1/2}  \Big)^2 \\
&\le \sum_K \frac{1}{|K|} \Big( \sum_{I_1 \subset K}^{(i_1)} |I_1| \Big) \Big( \sum_{I_1 \subset K}^{(i_1)} \int_{\R^m} \int_{I_1} |p^{i_1}_K f_y(x)|^2\,dx\,dy \Big) \\
&\le \sum_K \int_{\R^m} \int_{\R^n}  |p^{i_1}_K f_y(x)|^2\,dx\,dy \\
&= \int_{\R^m} \|f_y\|_2^2\,dy = \|f\|_2^2,
\end{align*}
where we again utilized orthogonality.
\end{proof}
We end this this section by concluding that
\begin{align*}
\mathop{\sum_{\ell(I_1) \le \ell(I_2)}}_{d(I_1,I_2) > \ell(I_1)^{\gamma_n}\ell(I_2)^{1-\gamma_n}}&  \sum_{J_1 \subsetneq J_2} \langle u_{J_2}\rangle_{J_1} \langle T(h_{I_1} \otimes u_{J_1}), h_{I_2} \otimes 1\rangle \langle f, h_{I_1} \otimes u_{J_1}\rangle
\langle g, h_{I_2} \otimes u_{J_2} \rangle \\
&= C\sum_{i_2 \ge 1} \sum_{i_1 \ge i_2} 2^{-i_1\delta/2} \langle S^{i_1i_200}f, g\rangle.
\end{align*}
\section{Separated/equal}
There holds that
\begin{displaymath}
|\langle T(h_{I_1} \otimes u_V), h_{I_2} \otimes u_V\rangle | \lesssim  \frac{|I_1|^{1/2}|I_2|^{1/2}}{|K|}\Big(\frac{\ell(I_1)}{\ell(K)}\Big)^{\delta/2}.
\end{displaymath}
Indeed, to see this, first estimate
\begin{align*}
|\langle T(h_{I_1} \otimes u_V), h_{I_2} \otimes u_V\rangle | \le |V|^{-1} \Big[ \mathop{\sum_{V', V'' \in \textup{ch}(V)}}_{V' \ne V''} & |\langle T(h_{I_1} \otimes \chi_{V'}), h_{I_2} \otimes \chi_{V''} \rangle| \\
&+ \sum_{V' \in \textup{ch}(V)}  |\langle T(h_{I_1} \otimes \chi_{V'}), h_{I_2} \otimes \chi_{V'} \rangle| \Big]. 
\end{align*}
We have by the kernel representation in $\R^n$ that
\begin{align*}
 |\langle T(h_{I_1} \otimes \chi_{V'}), h_{I_2} \otimes \chi_{V'} \rangle| \le C(\chi_{V'}, \chi_{V'}) &\frac{|I_1|^{1/2}|I_2|^{1/2}}{|K|}\Big(\frac{\ell(I_1)}{\ell(K)}\Big)^{\delta/2} \\ 
&\lesssim |V| \frac{|I_1|^{1/2}|I_2|^{1/2}}{|K|}\Big(\frac{\ell(I_1)}{\ell(K)}\Big)^{\delta/2}.
\end{align*}
For $V' \ne V''$ the estimate
\begin{displaymath}
|\langle T(h_{I_1} \otimes \chi_{V'}), h_{I_2} \otimes \chi_{V''} \rangle| \lesssim |V| \frac{|I_1|^{1/2}|I_2|^{1/2}}{|K|}\Big(\frac{\ell(I_1)}{\ell(K)}\Big)^{\delta/2}
\end{displaymath}
follows from the full kernel representation using the mixed H\"older and size estimate of $K$.

We may thus immediately write that
\begin{align*}
\mathop{\sum_{\ell(I_1) \le \ell(I_2)}}_{d(I_1,I_2) > \ell(I_1)^{\gamma_n}\ell(I_2)^{1-\gamma_n}}&  \sum_V \langle T(h_{I_1} \otimes u_V), h_{I_2} \otimes u_V \rangle \langle f, h_{I_1} \otimes u_V \rangle \langle g, h_{I_2} \otimes u_V \rangle \\
&= C\sum_{i_2 \ge 1} \sum_{i_1 \ge i_2} 2^{-i_1\delta/2} \langle S^{i_1i_200}f, g\rangle,
\end{align*}
where in this case $S^{i_1i_200}$ are cancellative shifts.
\section{Separated/nearby}
For the $J_1$ and $J_2$ in the nearby summation it is evident that $V = J_1 \vee J_2$ satisfies $\ell(V) \le 2^r\ell(J_1)$. Thus, we may write
\begin{align*}
&\mathop{\sum_{\ell(I_1) \le \ell(I_2)}}_{d(I_1,I_2) > \ell(I_1)^{\gamma_n}\ell(I_2)^{1-\gamma_n}} \mathop{\mathop{\sum_{\ell(J_1) \le \ell(J_2)}}_{d(J_1,J_2) \le \ell(J_1)^{\gamma_m}\ell(J_2)^{1-\gamma_m}}}_{J_1 \cap J_2 = \emptyset} \\
&= \sum_{i_2 \ge 1} \sum_{i_1 \ge i_2} \sum_{j_1 = 1}^r \sum_{j_2=1}^{j_1} \sum_K \sum_V 
\mathop{\mathop{\sum_{d(I_1,I_2) > \ell(I_1)^{\gamma_n}\ell(I_2)^{1-\gamma_n}}}_{I_1 \vee I_2 = K}}_{\ell(I_1) = 2^{-i_1}\ell(K), \,\ell(I_2) = 2^{-i_2}\ell(K)}
\mathop{\mathop{\sum_{d(J_1,J_2) \le \ell(J_1)^{\gamma_m}\ell(J_2)^{1-\gamma_m}, \, J_1 \cap J_2 = \emptyset}}_{J_1 \vee J_2 = V}}_{\ell(J_1) = 2^{-j_1}\ell(V), \,\ell(J_2) = 2^{-j_2}\ell(V)}.
\end{align*}
It is easy to get the required estimate
\begin{displaymath}
|\langle T(h_{I_1} \otimes u_{J_1}), h_{I_2} \otimes u_{J_2}\rangle | \lesssim  \frac{|I_1|^{1/2}|I_2|^{1/2}}{|K|}\Big(\frac{\ell(I_1)}{\ell(K)}\Big)^{\delta/2}
\end{displaymath}
by using the full kernel representation and the mixed H\"older and size estimate of $K$. Therefore, we are able to realize this part in the form
\begin{align*}
 C \sum_{i_2 \ge 1} \sum_{i_1 \ge i_2} \sum_{j_1 = 1}^r \sum_{j_2=1}^{j_1} 2^{-i_1\delta/2}2^{-j_1\delta/2} \langle S^{i_1i_2j_1j_2}f, g\rangle.
\end{align*}
\section{Inside/Inside}
We decompose
\begin{align*}
\langle T(h_{I_1} \otimes u_{J_1}), h_{I_2} \otimes u_{J_2}\rangle &= \langle T(h_{I_1} \otimes u_{J_1}), s_{I_2I_2} \otimes s_{J_1J_2}\rangle \\
&+ \langle u_{J_2}\rangle_{J_1} \langle T(h_{I_1} \otimes u_{J_1}), s_{I_2I_2} \otimes 1 \rangle \\
&+ \langle h_{I_2}\rangle_{I_1} \langle T(h_{I_1} \otimes u_{J_1}), 1 \otimes s_{J_1J_2} \rangle \\
&+  \langle h_{I_2}\rangle_{I_1} \langle u_{J_2}\rangle_{J_1} \langle T(h_{I_1} \otimes u_{J_1}), 1 \rangle,
\end{align*}
where $s_{I_1I_2} = \chi_{I_{2,1}^c}(h_{I_2} - \langle h_{I_2}\rangle_{I_{2,1}})$ and $s_{J_1J_2} = \chi_{J_{2,1}^c}[u_{J_2} - \langle u_{J_2} \rangle_{J_{2,1}}]$. The relevant properties are
spt$\,s_{I_1I_2} \subset I_{2,1}^c$,  spt$\,s_{J_1J_2} \subset J_{2,1}^c$, $|s_{I_1I_2}| \le 2 |I_2|^{-1/2}$ and $|s_{J_1J_2}| \le 2|J_2|^{-1/2}$.
\begin{lem}
There holds
\begin{displaymath}
|\langle T(h_{I_1} \otimes u_{J_1}), s_{I_2I_2} \otimes s_{J_1J_2}\rangle| \lesssim \frac{|I_1|^{1/2}}{|I_2|^{1/2}} \Big( \frac{\ell(I_1)}{\ell(I_2)} \Big)^{\delta/2} \frac{|J_1|^{1/2}}{|J_2|^{1/2}} \Big( \frac{\ell(J_1)}{\ell(J_2)} \Big)^{\delta/2}.
\end{displaymath}
\end{lem}
\begin{proof}
In the case that $\ell(I_1) < 2^{-r}\ell(I_2)$ and $\ell(J_1) < 2^{-r}\ell(J_2)$ one may use the H\"older estimate of $K$.
In the case $2^{-r}\ell(I_2) \le \ell(I_1) (\le \ell(I_2))$ and $2^{-r}\ell(J_2) \le \ell(J_1) (\le \ell(J_2))$ one splits
\begin{align*}
\langle T(h_{I_1} \otimes u_{J_1}), s_{I_2I_2} \otimes s_{J_1J_2}\rangle &= \langle T(h_{I_1} \otimes u_{J_1}), (\chi_{3I_1}s_{I_2I_2}) \otimes (\chi_{3J_1}s_{J_1J_2})\rangle \\
&+ \langle T(h_{I_1} \otimes u_{J_1}), (\chi_{3I_1}s_{I_2I_2}) \otimes (\chi_{(3J_1)^c}s_{J_1J_2})\rangle \\
&+ \langle T(h_{I_1} \otimes u_{J_1}), (\chi_{(3I_1)^c}s_{I_2I_2}) \otimes (\chi_{3J_1}s_{J_1J_2})\rangle \\
&+ \langle T(h_{I_1} \otimes u_{J_1}), (\chi_{(3I_1)^c}s_{I_2I_2}) \otimes (\chi_{(3J_1)^c}s_{J_1J_2})\rangle.
\end{align*}
The first term is controlled by the size estimate of the full kernel:
\begin{align*}
&|\langle T(h_{I_1} \otimes u_{J_1}), (\chi_{3I_1}s_{I_2I_2}) \otimes (\chi_{3J_1}s_{J_1J_2})\rangle| \\
&\le |I_1|^{-1/2} |I_2|^{-1/2} \int_{I_1} \int_{3I_1 \setminus I_1} \frac{dx_1\,dy_1}{|x_1-y_1|^n} \cdot |J_1|^{-1/2}|J_2|^{-1/2} \int_{J_1} \int_{3J_1 \setminus J_1} \frac{dx_2\,dy_2}{|x_2-y_2|^m} \\
&\lesssim \frac{|I_1|^{1/2}}{|I_2|^{1/2}} \frac{|J_1|^{1/2}}{|J_2|^{1/2}} \lesssim \frac{|I_1|^{1/2}}{|I_2|^{1/2}} \Big( \frac{\ell(I_1)}{\ell(I_2)} \Big)^{\delta/2} \frac{|J_1|^{1/2}}{|J_2|^{1/2}} \Big( \frac{\ell(J_1)}{\ell(J_2)} \Big)^{\delta/2}.
\end{align*}
The two terms after that are controlled using the mixed size and Hölder estimates of $K$. The last term is controlled using the H\"older estimate of $K$.
The mixed cases where $2^{-r}\ell(I_2) \le \ell(I_1) (\le \ell(I_2))$ and $\ell(J_1) < 2^{-r}\ell(J_2)$ or $\ell(I_1) < 2^{-r}\ell(I_2)$ and $2^{-r}\ell(J_2) \le \ell(J_1) (\le \ell(J_2))$ are handled similarly.
\end{proof}

The above lemma shows that
\begin{displaymath}
\sum_{I_1 \subsetneq I_2} \sum_{J_1 \subsetneq J_2} \langle T(h_{I_1} \otimes u_{J_1}), s_{I_2I_2} \otimes s_{J_1J_2}\rangle \langle f, h_{I_1} \otimes u_{J_1}\rangle \langle g, h_{I_2} \otimes u_{J_2}\rangle 
\end{displaymath}
can be realized in the form
\begin{displaymath}
C\sum_{i_1 = 1}^{\infty} \sum_{j_1 = 1}^{\infty} 2^{-i_1\delta/2} 2^{-j_1\delta/2} \langle S^{i_10j_10}f,g\rangle.
\end{displaymath}

The part
\begin{displaymath}
\sum_{I_1 \subsetneq I_2} \sum_{J_1 \subsetneq J_2} \langle u_{J_2}\rangle_{J_1} \langle T(h_{I_1} \otimes u_{J_1}), s_{I_2I_2} \otimes 1 \rangle  \langle f, h_{I_1} \otimes u_{J_1}\rangle \langle g, h_{I_2} \otimes u_{J_2}\rangle 
\end{displaymath}
can be written in the form
\begin{displaymath}
C\sum_{i_1 = 1}^{\infty} 2^{-i_1\delta/2} \langle S^{i_1000}f,g\rangle,
\end{displaymath}
where
\begin{displaymath}
S^{i_1000}f = \sum_K \mathop{\sum_{I_1 \subset K}}_{\ell(I_1) = 2^{-i_1}\ell(K)} h_K \otimes \Pi^*_{b_{I_1K}}(\langle f, h_{I_1} \rangle_1)
\end{displaymath}
and $b_{I_1K} = \langle T^*(s_{I_1K} \otimes 1), h_{I_1}\rangle_1 / C2^{-i_1\delta/2}$. Since one can check $\|b_{I_1K}\|_{\textup{BMO}(\R^m)} \le c|I_1|^{1/2}/|K|^{1/2}$, it is similarly as has already been
done in the separated/inside case seen that $\|S^{i_1000}f\|_2 \le \|f\|_2$. The proof of the BMO estimate is similar to the proof of the previous lemma.

Completely analogously one can write
\begin{displaymath}
\sum_{I_1 \subsetneq I_2} \sum_{J_1 \subsetneq J_2} \langle h_{I_2}\rangle_{I_1} \langle T(h_{I_1} \otimes u_{J_1}), 1 \otimes s_{J_1J_2} \rangle  \langle f, h_{I_1} \otimes u_{J_1}\rangle \langle g, h_{I_2} \otimes u_{J_2}\rangle 
\end{displaymath}
in the form
\begin{displaymath}
C\sum_{j_1 = 1}^{\infty} 2^{-j_1\delta/2} \langle S^{00j_10}f,g\rangle,
\end{displaymath}
where $S^{00j_10}$ is a non-cancellative $L^2$ bounded shift.

The last part
\begin{displaymath}
\sum_{I_1 \subsetneq I_2} \sum_{J_1 \subsetneq J_2} \langle h_{I_2}\rangle_{I_1} \langle u_{J_2}\rangle_{J_1} \langle T(h_{I_1} \otimes u_{J_1}), 1 \rangle \langle f, h_{I_1} \otimes u_{J_1}\rangle \langle g, h_{I_2} \otimes u_{J_2}\rangle 
\end{displaymath}
collapses to
\begin{align*}
\sum_{K,V} \langle g \rangle_{K \times V} \langle T^*1, h_K \otimes u_V \rangle \langle f, h_K \otimes u_V\rangle = C\langle \Pi_{T^*1/C}^*f, g\rangle, 
\end{align*}
where
\begin{displaymath}
\Pi_bf = \sum_{K,V} \langle f \rangle_{K \times V} \langle b, h_K \otimes u_V\rangle h_K \otimes u_V
\end{displaymath}
is a bounded shift of the type $(0,0,0,0)$ for b in the product BMO of $\R^n \times \R^m$. So here we can set $S^{0000} = \Pi_{T^*1/C}^*$. Note that the correct normalization for this shift would follow
just from the various kernel estimates and the weak boundedness property.
\begin{rem}\label{different par}
In the proof of this representation theorem there are paraproducts of essentially three different types. We have seen two types already: the full paraproduct
\begin{displaymath}
\Pi_bf = \sum_{K,V} \langle f \rangle_{K \times V} \langle b, h_K \otimes u_V\rangle h_K \otimes u_V
\end{displaymath}
and some half paraproducts, like
\begin{displaymath}
f \mapsto \sum_K \mathop{\sum_{I_1 \subset K}}_{\ell(I_1) = 2^{-i_1}\ell(K)} h_K \otimes \Pi^*_{b_{I_1K}}(\langle f, h_{I_1} \rangle_1),
\end{displaymath}
which have a paraproduct part only in the $\R^n$ or $\R^m$ variable. The third type of paraproduct does not surface in our current sum, where $\ell(I_1) \le \ell(I_2)$ and $\ell(J_1) \le \ell(J_2)$. However, for example in the mixed
case, where $\ell(I_1) \le \ell(I_2)$ and $\ell(J_1) > \ell(J_2)$, one has in the corresponding inside/inside part the mixed full paraproduct
\begin{align*}
f &\mapsto \sum_{K,V} |K \times V|^{-1} \langle T_1(1), h_K \otimes u_V\rangle \langle f, h_K \otimes \chi_V\rangle \chi_K \otimes u_V \\
&= \sum_{K,V} \langle T_1(1), h_K \otimes u_V\rangle \langle f, h_K \otimes u_V^2 \rangle h_K^2 \otimes u_V,
\end{align*} 
which is $L^2$ bounded as $T_1(1)$ belongs to the product BMO of $\R^n \times \R^m$ by assumption. It is rather straightforward and well-known that both of these full paraproducts are bounded on $L^2$
if they are defined by functions in the dyadic product BMO. This can be proven by duality -- see for example \cite{PV}.
\end{rem}
\section{Inside/equal}
One splits
\begin{align*}
\langle T(h_{I_1} \otimes u_V), h_{I_2} \otimes u_V\rangle &= \langle T(h_{I_1} \otimes u_V), s_{I_1I_2} \otimes u_V\rangle + 
\langle h_{I_2} \rangle_{I_1} \langle T(h_{I_1} \otimes u_V), 1 \otimes u_V\rangle,
\end{align*}
where $s_{I_1I_2} = \chi_{I_{2,1}^c}(h_{I_2} - \langle h_{I_2}\rangle_{I_{2,1}})$ satisfies spt$\,s_{I_1I_2} \subset I_{2,1}^c$ and $|s_{I_1I_2}| \le 2 |I_2|^{-1/2}$.

One may write
\begin{displaymath}
\sum_{I_1 \subsetneq I_2} \sum_V \langle T(h_{I_1} \otimes u_V), s_{I_1I_2} \otimes u_V\rangle \langle f, h_{I_1} \otimes u_V\rangle \langle g, h_{I_2} \otimes u_V\rangle
\end{displaymath}
in the form
\begin{displaymath}
C \sum_{i_1=1}^{\infty} 2^{-i_1\delta/2} \langle S^{i_1000}f, g\rangle
\end{displaymath}
with cancellative shifts.
For this one needs that
\begin{displaymath}
|\langle T(h_{I_1} \otimes u_V), s_{I_1I_2} \otimes u_V\rangle| \lesssim \frac{|I_1|^{1/2}}{|I_2|^{1/2}} \Big( \frac{\ell(I_1)}{\ell(I_2)} \Big)^{\delta/2}.
\end{displaymath}
Estimate
\begin{align*}
|\langle T(h_{I_1} \otimes u_V), s_{I_1I_2} \otimes u_V\rangle| \le |V|^{-1}\Big[ &\mathop{\sum_{V', V'' \in \textup{ch}(V)}}_{V' \ne V''} |\langle T(h_{I_1} \otimes \chi_{V'}), s_{I_1I_2} \otimes \chi_{V''}\rangle| \\
&+ \sum_{V' \in \textup{ch}(V)} |\langle T(h_{I_1} \otimes \chi_{V'}), s_{I_1I_2} \otimes \chi_{V'}\rangle|\Big].
\end{align*}
In the case $V' \ne V''$ use the full kernel representation. In the diagonal case use the kernel representation in $\R^n$.
If $\ell(I_1) < 2^{-r}\ell(I_2)$, use the mixed size and H\"older estimate of $K$ (in the case $V' \ne V''$) or the H\"older estimate for the kernel $K_{\chi_{V'}, \chi_{V'}}$ (in the case $V' = V''$).
In the case $2^{-r}\ell(I_2) \le \ell(I_1)$ split $s_{I_1I_2} = \chi_{3I_1}s_{I_1I_2} + \chi_{(3I_1)^c}s_{I_1I_2}$. For $V' \ne V''$ use the size estimate of $K$ for the first term
and the mixed size and H\"older estimate of $K$ for the second term. In the case $V' = V''$ use the size estimate of $K_{\chi_{V'}, \chi_{V'}}$ for the first term, and the
H\"older estimate of $K_{\chi_{V'}, \chi_{V'}}$ for the second term.

One writes
\begin{displaymath}
\sum_{I_1 \subsetneq I_2} \sum_V\langle h_{I_2} \rangle_{I_1} \langle T(h_{I_1} \otimes u_V), 1 \otimes u_V\rangle \langle f, h_{I_1} \otimes u_V\rangle \langle g, h_{I_2} \otimes u_V\rangle
\end{displaymath}
in the form
\begin{displaymath}
C\langle S^{0000}f,g\rangle,
\end{displaymath}
where in this case
\begin{align*}
S^{0000}f &= \sum_V \Pi_{b_V}^*(\langle f, u_V\rangle_2) \otimes u_V
\end{align*}
and $b_V = \langle T^*(1 \otimes u_V), u_V\rangle_2/C$. This is indeed a non-cancellative shift of the type $(0,0,0,0)$.
\begin{lem}
There holds $\|b_V\|_{\textup{BMO}(\R^n)} \le c$.
\end{lem}
\begin{proof}
Fix a cube $K \subset \R^n$ and a function $a$ so that spt$\,a \subset K$, $|a| \le 1$ and $\int a = 0$. We need to show that
$|\langle T(a \otimes u_V), 1 \otimes u_V\rangle| \lesssim |K|$. We begin with the split
\begin{align*}
\langle T(a \otimes u_V), 1 \otimes u_V\rangle = \langle T(a \otimes u_V), \chi_K \otimes u_V\rangle + \langle T(a \otimes & u_V), \chi_{3K \setminus K} \otimes u_V\rangle \\
&+ \langle T(a \otimes u_V), \chi_{(3K)^c} \otimes u_V\rangle.
\end{align*}

There holds
\begin{align*}
\langle T(a \otimes u_V), \chi_{3K \setminus K} \otimes u_V\rangle \le |V|^{-1} \Big[&\mathop{\sum_{V', V'' \in \textup{ch}(V)}}_{V' \ne V''}  |\langle T(a \otimes \chi_{V'}), \chi_{3K \setminus K} \otimes \chi_{V''} \rangle| \\
& + \sum_{V' \in \textup{ch}(V)} |\langle T(a \otimes \chi_{V'}), \chi_{3K \setminus K} \otimes \chi_{V'} \rangle|\Big],
\end{align*}
where
\begin{align*}
|\langle T(a \otimes \chi_{V'}), &\chi_{3K \setminus K} \otimes \chi_{V''} \rangle| \\
&\le \int_K \int_{3K \setminus K} \frac{1}{|x_1-y_1|^n}\,dx_1\,dy_1 \cdot \int_{V'} \int_{V''} \frac{1}{|x_2-y_2|^m}\,dx_2\,dy_2 \lesssim |K||V|
\end{align*}
and
\begin{align*} 
|\langle T(a \otimes \chi_{V'}), \chi_{3K \setminus K} \otimes \chi_{V'} \rangle|
&\le  C(\chi_{V'}, \chi_{V'}) \int_K \int_{3K \setminus K} \frac{1}{|x_1-y_1|^n}\,dx_1\,dy_1 \lesssim |K||V|.
\end{align*}
Furthermore, we have
\begin{align*}
|\langle T(a \otimes u_V), \chi_{(3K)^c} \otimes u_V\rangle| \le |V|^{-1} \Big[&\mathop{\sum_{V', V'' \in \textup{ch}(V)}}_{V' \ne V''} |\langle T(a \otimes \chi_{V'}), \chi_{(3K)^c} \otimes \chi_{V''} \rangle| \\
 & + \sum_{V' \in \textup{ch}(V)} |\langle T(a \otimes \chi_{V'}), \chi_{(3K)^c} \otimes \chi_{V'} \rangle|\Big],
\end{align*}
where
\begin{align*}
&|\langle T(a \otimes \chi_{V'}), \chi_{(3K)^c} \otimes \chi_{V''} \rangle|  \\
&\lesssim |K| \cdot \ell(K)^{\delta} \int_{(3K)^c} \frac{dx_1}{|x_1-c_K|^{n+\delta}} \cdot \int_{V'} \int_{V''}  \frac{1}{|x_2-y_2|^m}\,dx_2\,dy_2 \lesssim |K||V|
\end{align*}
and
\begin{align*}
|\langle T(a \otimes \chi_{V'}), \chi_{(3K)^c} \otimes \chi_{V'} \rangle|
&\le  C(\chi_{V'}, \chi_{V'})\int_K \int_{(3K)^c} \frac{\ell(K)^{\delta}}{|x_1-c_K|^{n+\delta}}\,dx_1\,dy_1 \lesssim |K||V|.
\end{align*}

For the first term we again begin with the estimate
\begin{displaymath}
|\langle T(a \otimes u_V), \chi_K \otimes u_V\rangle| \lesssim |V|^{-1} \sum_{V', V'' \in \textup{ch}(V)} |\langle T(a \otimes \chi_{V'}), \chi_K \otimes \chi_{V''}\rangle|.
\end{displaymath}
Let us consider the case $V' \ne V''$. In this case we have
\begin{align*}
|\langle T(a \otimes \chi_{V'}), \chi_K \otimes \chi_{V''}\rangle| &= \Big| \int_{V'} \int_{V''} K_{a, \chi_K}(x_2,y_2)\,dx_2\,dy_2\Big| \\
&\le C(a, \chi_K) \int_{V'} \int_{V''} \frac{1}{|x_2-y_2|^m}\,dx_2\,dy_2 \lesssim |K| |V|.
\end{align*}
Thus, we are only left with the need for the estimate $|\langle T(a \otimes \chi_{V'}), \chi_K \otimes \chi_{V'}\rangle| \lesssim |K||V|$ -- but this is one of the diagonal BMO assumptions.

\end{proof}
Because of this lemma, one can show, similarly but with a bit less effort than in Proposition \ref{singlepara}, that $S^{0000}$ is $L^2$ bounded.
\section{Inside/nearby}
This goes very much so in the same vein as the inside/equal case. In fact, this is easier since the nearby cubes do not intersect by definition.
From the series with the matrix element
$\langle T(h_{I_1} \otimes u_{J_1}), s_{I_1I_2} \otimes u_{J_2}\rangle$ we get
\begin{displaymath}
C\sum_{i_1 = 1}^{\infty} \sum_{j_1 = 1}^r \sum_{j_2 = 1}^{j_1} 2^{-i_1\delta/2}2^{-j_1\delta/2} \langle S^{i_10j_1j_2}f, g\rangle.
\end{displaymath}
From the series with the matrix element $\langle h_{I_2}\rangle_{I_1} \langle T(h_{I_1} \otimes u_{J_1}), 1 \otimes u_{J_2}\rangle$ we get
\begin{displaymath}
C \sum_{j_1 = 1}^r \sum_{j_2 = 1}^{j_1} 2^{-j_1\delta/2} \langle S^{00j_1j_2}f, g\rangle
\end{displaymath}
with bounded non-cancellative shifts.

\section{Equal/equal}
This part can be realized in the form $C\langle S^{0000}f, g\rangle$ for a cancellative shift, since one can just
estimate $|\langle T(h_K \otimes u_V), h_K \otimes u_V\rangle| \lesssim 1$. This estimate is an easy consequence of the weak boundedness property and the size estimates
of our kernels.
\section{Equal/nearby}
This part is clearly of the form
\begin{displaymath}
C  \sum_{j_1 = 1}^r \sum_{j_2=1}^{j_1} 2^{-j_1\delta/2} \langle S^{00j_1j_2}f,g\rangle,
\end{displaymath}
where the shifts are cancellative. Here one can again just use the estimate $|\langle T(h_K \otimes u_{J_1}), h_K \otimes u_{J_2}\rangle| \lesssim 1$, which follows
just from the size estimates of our kernels.
\section{Nearby/Nearby}
This part is of the form
\begin{displaymath}
C\sum_{i_1 = 1}^r \sum_{i_2=1}^{i_1}  \sum_{j_1 = 1}^r \sum_{j_2=1}^{j_1} 2^{-i_1\delta/2}2^{-j_1\delta/2} \langle S^{i_1i_2j_1j_2}f,g\rangle
\end{displaymath}
once again because of the easy estimate $|\langle T(h_{I_1} \otimes u_{J_1}), h_{I_2} \otimes u_{J_2}\rangle| \lesssim 1$. This follows from the size estimate for the full kernel.

\end{document}